\documentclass{amsart}

\usepackage[T1]{fontenc}
\usepackage{lmodern}
\usepackage{amsmath,amssymb,amsthm,mathtools}
\usepackage[
  backend=biber,
  style=alphabetic,
  sorting=nyt
]{biblatex}

\newtheorem{theorem}{Theorem}
\newtheorem{proposition}[theorem]{Proposition}
\newtheorem{lemma}[theorem]{Lemma}
\newtheorem{corollary}[theorem]{Corollary}

\newtheorem{example}[theorem]{Example}

\theoremstyle{definition}
\newtheorem{definition}[theorem]{Definition}

\theoremstyle{remark}
\newtheorem{remark}[theorem]{Remark}

\newcommand{\Spec}[1]{\operatorname{Spec}\left(#1\right)}
\newcommand{\nat}{\mathbb N}
\newcommand{\reals}{\mathbb R}
\newcommand{\integers}{\mathbb Z}

\newcommand{\Op}[1]{\mathcal{O}\left(#1\right)}
\newcommand{\cont}[1]{C\left(#1\right)}

\newcommand{\KOp}[1]{\mathring{\mathcal{K}}\left(#1\right)}
\newcommand{\linfty}{\ell^\infty}

\title[Compact-open subsets in the spectrum of $\cont{T}$]{The Algebra of compact-open subsets in the spectrum
of the ring $\cont{T}$ for an infinite compact Hausdorff space $T$}

\author{Evgeny Kuznetsov}

\address{Department of Mathematical logic, 
Razmadze Mathematical Institute, 
2 Merab Aleksidze II Lane, Tbilisi 0193, Georgia}

\email{e.kuznetsov@freeuni.edu.ge}

\date{July 2026}


\begin{document}

\begin{abstract}
For a commutative B\'ezout ring $R$, we give a criterion --- in terms of
colon ideals with principal radical --- for the lattice $\KOp{\Spec R}$
of compact open subsets of $\Spec R$ to be a Heyting algebra.
Bezhanishvili and Tressl showed that $\KOp{\Spec{\cont T}}$ is
pseudocomplemented whenever $T$ is a basically disconnected compact
Hausdorff space, and asked whether $\Spec{\cont{\beta\nat}}$ is actually
an Esakia space. We show it is not: applying our criterion to
$\cont{\beta\nat}\cong\linfty(\nat,\reals)$ produces a diagonal
counterexample, and the same obstruction rules out $\beta D$ for every
infinite discrete $D$. A grid-existence theorem for $\sigma$-complete
Boolean algebras lets us push the construction to every basically
disconnected compact Hausdorff space, settling the Bezhanishvili--Tressl
question completely: for no infinite compact Hausdorff space $T$ is
$\Spec{\cont T}$ an Esakia space.
\end{abstract}


\maketitle

\section{Problem and status}
\label{sec:problem}

Let
\[
A=\cont{\beta\nat},
\qquad
X=\Spec{A}.
\]
The problem is to determine whether the lattice
\[
\KOp{X}=\{U\subseteq X:U\text{ is compact and open}\}
\]
is a Heyting algebra. Equivalently, one asks whether the spectral space
$X$ is Esakia.

Bezhanishvili and Tressl proved that, whenever $T$ is basically
disconnected, the lattice $\KOp{\Spec{\cont{T}}}$ is pseudocomplemented,
and explicitly left open whether $\Spec{\cont{\beta\nat}}$
is Esakia; see \cite{bezhanishvili-tressl-2026}.

Here is the plan. We first set up the standard reduction, then run a
diagonal construction that turns it into a negative answer;
Theorem~\ref{claim:negative-answer} does this for $A$ itself.
Corollary~\ref{cor:beta-D} shows the same obstruction already rules out
\emph{every} space of the form $\beta D$ with $D$ infinite discrete, not
just $D=\nat$. Section~\ref{sec:full-negative} then closes the remaining
gap: generalizing the diagonal argument to arbitrary basically
disconnected spaces (Section~\ref{sec:beyond-betaD}) and proving that a
countable grid always exists, we answer the Bezhanishvili--Tressl
question completely --- for no infinite compact Hausdorff space $T$ is
the lattice $\KOp{\Spec{\cont{T}}}$ a Heyting algebra.

\section{Algebraic model and compact opens}

\subsection*{Preliminaries on prime spectra}

Recall that $\Spec{R}$ denotes the set of prime ideals of a commutative
ring $R$, and that for $f\in R$,
\[
D(f)=\{\mathfrak p\in\Spec{R}:f\notin\mathfrak p\}
\]
is the associated basic open set; these sets form a basis for the
topology on $\Spec{R}$. For an ideal $I\subseteq R$, the \emph{radical}
of $I$ is
\[
\sqrt I=\{x\in R:x^n\in I\text{ for some }n\geq1\},
\]
and, for $f\in R$, the \emph{colon ideal} (or ideal quotient) of $I$ by
$f$ is
\[
(I:f)=\{a\in R:af\in I\}
\]
(this is again an ideal of $R$). Some standard facts about $D(-)$ are used
repeatedly below.

\begin{lemma}\label{lem:D-intersection}
For $a,f\in R$, $D(a)\cap D(f)=D(af)$.
\end{lemma}

\begin{proof}
Fix a prime $\mathfrak p$. If $af\notin\mathfrak p$, then $a\in\mathfrak
p$ would force $af=a\cdot f\in\mathfrak p$ (ideals absorb multiplication
by ring elements), so $a\notin\mathfrak p$; likewise $f\notin\mathfrak
p$. Conversely, if $a\notin\mathfrak p$ and $f\notin\mathfrak p$,
primeness of $\mathfrak p$ --- in the contrapositive form $x\notin
\mathfrak p,\,y\notin\mathfrak p\Rightarrow xy\notin\mathfrak p$ --- gives
$af\notin\mathfrak p$. So $af\notin\mathfrak p\iff a\notin\mathfrak
p\text{ and }f\notin\mathfrak p$, i.e.\ $\mathfrak p\in D(af)\iff
\mathfrak p\in D(a)\cap D(f)$.
\end{proof}

\begin{lemma}\label{lem:D-inclusion}
For $x,y\in R$, $D(x)\subseteq D(y)$ if and only if $x\in\sqrt{(y)}$.
\end{lemma}

\begin{proof}
$D(x)\subseteq D(y)$ says every prime not containing $x$ fails to
contain $y$; contrapositively, every prime containing $y$ contains $x$,
i.e.\ $x\in\bigcap_{\mathfrak p\supseteq(y)}\mathfrak p$. It is a
standard fact that $\sqrt{(y)}=\bigcap_{\mathfrak p\supseteq(y)}\mathfrak
p$: if $x\in\sqrt{(y)}$ then $x^n\in(y)$ for some $n$, and any prime
$\mathfrak p\supseteq(y)$ contains $x^n$, hence contains $x$ (primes
absorb powers); conversely, if no power of $x$ lies in $(y)$, the
multiplicative set $\{1,x,x^2,\dots\}$ is disjoint from $(y)$, so $(y)$
extends to a prime ideal still disjoint from that set (a standard
Zorn's-lemma argument), exhibiting a prime containing $y$ but not $x$.
The claim follows.
\end{proof}

For an ideal $I\subseteq R$, write
\[
D(I)=\bigcup_{a\in I}D(a).
\]

\begin{lemma}\label{lem:D-equality}
For ideals $I,J\subseteq R$, $D(I)=D(J)$ if and only if $\sqrt I=\sqrt J$.
\end{lemma}

\begin{proof}
The argument of Lemma~\ref{lem:D-inclusion} did not use that $(y)$ was
principal: for any ideal $J$ and any $a\in R$, the same reasoning gives
$D(a)\subseteq D(J)\iff a\in\sqrt J$.

If $\sqrt I=\sqrt J$: for $a\in I\subseteq\sqrt I=\sqrt J$, this gives
$D(a)\subseteq D(J)$; taking the union over $a\in I$ gives
$D(I)\subseteq D(J)$, and symmetrically $D(J)\subseteq D(I)$, so
$D(I)=D(J)$.

Conversely, if $D(I)=D(J)$: for $a\in I$, $D(a)\subseteq D(I)=D(J)$, so
$a\in\sqrt J$; hence $I\subseteq\sqrt J$, and since taking radicals is
monotone and idempotent, $\sqrt I\subseteq\sqrt{\sqrt J}=\sqrt J$.
Symmetrically $\sqrt J\subseteq\sqrt I$, so $\sqrt I=\sqrt J$.
\end{proof}

\begin{definition}
  For a set $D$, 
  the space $\linfty{(D,\reals)}$ is the space given by
\[\linfty{(D,\reals)} = 
\left\{
f\in \reals^D
: 
\sup_{d\in D}|f(d)|<\infty
\right\}\]
\end{definition}

Restriction to the dense discrete subspace $\nat\subseteq\beta\nat$ gives the
standard isomorphism
\[
\cont{\beta\nat}\cong\linfty(\nat,\reals);
\]
see \cite[6.5(II), 6.6(b)]{GillmanJerison} for the general fact underlying
this (for a Tychonoff space $X$, restriction identifies $\cont{\beta X}$
with the ring $C^\ast(X,\reals)$ of bounded continuous functions on $X$),
specialized here to the discrete space $X=\nat$, where every function is
continuous, so $C^\ast(\nat,\reals)=\linfty(\nat,\reals)$.

\begin{lemma}\label{lem:beta-nat-iso}
The restriction map $\rho:\cont{\beta\nat}\to\linfty(\nat,\reals)$,
$f\mapsto f|_\nat$, is a ring isomorphism.
\end{lemma}

\begin{proof}
$\rho$ is well defined and a ring homomorphism: $\beta\nat$ is compact, so
any continuous $f:\beta\nat\to\reals$ has bounded (indeed compact) image,
hence $f|_\nat\in\linfty(\nat,\reals)$; restriction preserves pointwise
sums, products, and the constant $1$.

$\rho$ is injective. If $f,g\in\cont{\beta\nat}$ agree on $\nat$,
the set $\{x\in\beta\nat:f(x)=g(x)\}$ is the preimage of the diagonal
$\Delta\subseteq\reals\times\reals$ under the continuous map
$(f,g):\beta\nat\to\reals\times\reals$; since $\reals$ is Hausdorff,
$\Delta$ is closed, so this equalizer is a closed subset of $\beta\nat$
containing the dense subset $\nat$, hence equals $\beta\nat$. So $f=g$.

$\rho$ is surjective. This is the universal property characterizing
$\beta\nat$: every continuous map from $\nat$ into a compact Hausdorff
space extends (uniquely) to a continuous map from $\beta\nat$. Given
$a\in\linfty(\nat,\reals)$, view $a$ as a continuous map
$\nat\to[-\|a\|_\infty,\|a\|_\infty]$ (a compact Hausdorff space; every
map out of the discrete space $\nat$ is continuous); the extension
property gives $f\in\cont{\beta\nat}$ with $f|_\nat=a$, i.e.\
$\rho(f)=a$.

A bijective ring homomorphism is a ring isomorphism.
\end{proof}

From now on we identify $A$ with $\linfty(\nat,\reals)$, via
Lemma~\ref{lem:beta-nat-iso}. In particular, every bounded real sequence
represents an element of $A$.

\begin{lemma}\label{lem:bezout-linfty}
Every finitely generated ideal of $A=\linfty(\nat,\reals)$ is principal.
More precisely, for $a_1,\dots,a_n\in A$, if
\[
k=|a_1|+\cdots+|a_n|,
\]
then
\[
(a_1,\dots,a_n)=(k).
\]
\end{lemma}

\begin{proof}
For each $i$, define the bounded sequence
\[
u_i(m)=
\begin{cases}
\dfrac{a_i(m)}{k(m)},&k(m)\neq0,\\[1ex]
0,&k(m)=0.
\end{cases}
\]
Since $|u_i(m)|\leq1$, one has $u_i\in A$, and $a_i=u_i k$. Hence
\[
(a_1,\dots,a_n)\subseteq(k).
\]
Conversely, for each $i$ define
\[
s_i(m)=
\begin{cases}
\dfrac{|a_i(m)|}{a_i(m)},&a_i(m)\neq0,\\[1ex]
0,&a_i(m)=0.
\end{cases}
\]
Then $s_i\in A$ and $|a_i|=s_i a_i$. Thus $k\in(a_1,\dots,a_n)$, so
\[
(k)\subseteq(a_1,\dots,a_n).
\]
\end{proof}

\begin{lemma}\label{lem:compact-opens-principal}
Every compact open subset of $\Spec{A}$ is of the form $D(f)$ for some
$f\in A$.
\end{lemma}

\begin{proof}
Every compact open subset of an affine spectrum is a finite union of standard
opens. Hence it has the form
\[
D(f_1)\cup\cdots\cup D(f_n)=D((f_1,\dots,f_n)).
\]
By Lemma~\ref{lem:bezout-linfty}, the ideal $(f_1,\dots,f_n)$ is principal,
say $(f_1,\dots,f_n)=(f)$. Therefore the compact open is $D(f)$.
\end{proof}

Thus
\[
\KOp{X}=\{D(f):f\in A\}.
\]

\section{Implication and the general B\'ezout criterion}

We first record the relation between implication in all opens and implication
among compact opens.

\begin{lemma}\label{lem:implications-coincide}
Let $Y$ be a spectral space. If $\KOp{Y}$ is a Heyting algebra, then for all
$U,V\in\KOp{Y}$,
\[
U\Rightarrow_{\KOp{Y}}V
=
U\Rightarrow_{\Op{Y}}V.
\]
In particular, the implication computed in $\Op{Y}$ is compact open.
\end{lemma}

\begin{proof}
Put
\[
W=U\Rightarrow_{\KOp{Y}}V,
\qquad
H=U\Rightarrow_{\Op{Y}}V.
\]
Since $W\cap U\subseteq V$, the defining adjunction in $\Op{Y}$ gives
$W\subseteq H$.

Conversely, let $y\in H$. Compact opens form a basis of a spectral space, so
there exists $C\in\KOp{Y}$ with
\[
y\in C\subseteq H.
\]
Then $C\cap U\subseteq V$, and the defining adjunction in $\KOp{Y}$ gives
$C\subseteq W$. Hence $y\in W$. Therefore $H\subseteq W$.
\end{proof}

\begin{proposition}\label{prop:frame-implication}
Let $R$ be a commutative ring and let $Y=\Spec{R}$. For all $f,g\in R$,
\[
D(f)\Rightarrow_{\Op{Y}}D(g)
=
D\bigl((\sqrt{(g)}:f)\bigr),
\]
where
\[
(\sqrt{(g)}:f)=\{a\in R:af\in\sqrt{(g)}\}.
\]
\end{proposition}

\begin{proof}
Since the standard opens form a basis, the implication is the union of all
$D(a)$ such that
\[
D(a)\cap D(f)\subseteq D(g).
\]
For $a\in R$,
\[
\begin{aligned}
D(a)\cap D(f)\subseteq D(g)
&\iff D(af)\subseteq D(g)\\
&\iff af\in\sqrt{(g)}\\
&\iff a\in(\sqrt{(g)}:f).
\end{aligned}
\]
(The first equivalence is Lemma~\ref{lem:D-intersection}; the second is
Lemma~\ref{lem:D-inclusion}; the third is the definition of the colon
ideal.) Taking the union over all such $a$ yields the formula.
\end{proof}

\begin{definition}\label{def:bezout-ring}
A commutative ring $R$ is a \emph{B\'ezout ring} if every finitely
generated ideal of $R$ is principal.
\end{definition}

\begin{theorem}\label{thm:bezout-criterion}
Let $R$ be a commutative B\'ezout ring. Then $\KOp{\Spec{R}}$ is a Heyting
algebra if and only if, for every $f,g\in R$, the ideal
\[
(\sqrt{(g)}:f)
\]
has principal radical; that is, there exists $k\in R$ such that
\[
\sqrt{(\sqrt{(g)}:f)}=\sqrt{(k)}.
\]
\end{theorem}

\begin{proof}
Put $Y=\Spec{R}$. Since $R$ is B\'ezout, every compact open subset of $Y$ is
principal.

Suppose first that $\KOp{Y}$ is Heyting. For $f,g\in R$,
Lemma~\ref{lem:implications-coincide} and
Proposition~\ref{prop:frame-implication} show that
\[
D\bigl((\sqrt{(g)}:f)\bigr)
\]
is compact open. Hence it equals $D(k)$ for some $k\in R$. By
Lemma~\ref{lem:D-equality},
\[
D(I)=D(J)\iff\sqrt I=\sqrt J,
\]
so we obtain
\[
\sqrt{(\sqrt{(g)}:f)}=\sqrt{(k)}.
\]

Conversely, assume that every ideal $(\sqrt{(g)}:f)$ has principal radical.
Let $U,V\in\KOp{Y}$. Write $U=D(f)$ and $V=D(g)$. Choose $k\in R$ such that
\[
\sqrt{(\sqrt{(g)}:f)}=\sqrt{(k)}.
\]
Then Proposition~\ref{prop:frame-implication} gives
\[
U\Rightarrow_{\Op{Y}}V
=D\bigl((\sqrt{(g)}:f)\bigr)
=D(k)\in\KOp{Y}.
\]
For every $W\in\KOp{Y}$,
\[
W\subseteq D(k)
\iff
W\cap U\subseteq V.
\]
Thus $D(k)$ satisfies the defining adjunction for implication in $\KOp{Y}$.
Therefore $\KOp{Y}$ is a Heyting algebra.
\end{proof}

\begin{remark}
The phrase ``has principal radical'' means
\[
\sqrt I=\sqrt{(k)}
\]
for some $k$; it does not assert that the radical ideal $\sqrt I$ itself is
principal.
\end{remark}

Applying Theorem~\ref{thm:bezout-criterion} to
$A=\cont{\beta\nat}\cong\linfty(\nat,\reals)$ reduces the original problem to
whether every ideal
\[
(\sqrt{(g)}:f)
\]
has principal radical.

\section{Proposed counterexample}
\label{sec:proposed-counterexample}

Fix a bijection
\[
\nat\cong\nat_{>0}\times\nat_{>0}
\]
and write elements of $A$ as bounded double sequences $a(j,t)$.
Define
\[
f(j,t)=e^{-t},
\qquad
g(j,t)=e^{-jt},
\]
and put
\[
J=(\sqrt{(g)}:f).
\]
For $j\geq1$, let
\[
B_j=\{j\}\times\nat_{>0},
\qquad
e_j=\chi_{B_j}.
\]
Here $B_j\subseteq\nat_{>0}\times\nat_{>0}$ is the $j$-th row of the grid, and
$\chi_{B_j}\in A=\linfty(\nat,\reals)$ denotes its characteristic
function, i.e.\ the bounded sequence
\[
\chi_{B_j}(j',t)=
\begin{cases}
1 & j'=j,\\
0 & j'\neq j,
\end{cases}
\]
so $e_j$ is the element of $A$ equal to $1$ on row $j$ and $0$ on every
other row. Since $\chi_{B_j}$ takes only the values $0,1$, it is an
idempotent of $A$ ($e_j^2=e_j$); this is what ``row idempotent'' below
refers to. The same $\chi_{(-)}$ notation is used throughout for
characteristic functions of subsets of the relevant space (e.g.\
$\chi_S$, $\chi_{E_n}$, $\chi_{[b_j]}$).

\subsection{Row idempotents}

\begin{lemma}\label{lem:row-idempotents}
For every $j\geq1$, one has $e_j\in J$.
\end{lemma}

\begin{proof}
Since $g=f^j$ ($j$-th ring power, pointwise) on the $j$-th row,
\[
(e_jf)^j=e_jf^j=e_jg\in(g).
\]
(Here $f^j(j',t)=e^{-jt}$ for every row $j'$, since it is the $j$-th
power of the single element $f$; this coincides with $g(j',t)=e^{-j't}$
exactly when $j'=j$, i.e.\ exactly on row $j$ --- which is the only row
where $e_j$ is nonzero, so $e_jf^j=e_jg$ holds even though $f^j\neq g$
globally.)
Hence $e_jf\in\sqrt{(g)}$, and therefore $e_j\in J$.
\end{proof}

\subsection{Consequences of hypothetical compactness}

Assume, toward a contradiction, that $D(J)$ is compact. Since
\[
D(J)=\bigcup_{a\in J}D(a),
\]
there exist $a_1,\dots,a_n\in J$ such that
\[
D(J)=D(a_1)\cup\cdots\cup D(a_n).
\]
Set
\[
k=|a_1|+\cdots+|a_n|.
\]
By Lemma~\ref{lem:bezout-linfty},
\[
(a_1,\dots,a_n)=(k),
\]
and hence
\begin{equation}\label{eq:radical-k}
D(J)=D(k),
\qquad
\sqrt J=\sqrt{(k)}.
\end{equation}

\begin{lemma}\label{lem:row-lower-bound}
For every $j\geq1$,
\[
\delta_j:=\inf_{t\geq1}|k(j,t)|>0.
\]
\end{lemma}

\begin{remark}
In other words, along any single row $B_j$, $k(j,t)$ never gets
arbitrarily close to $0$ as $t$ varies --- it stays at least $\delta_j$
away from $0$, uniformly in $t$. The bound $\delta_j$ is allowed to
depend on (and shrink with) $j$; only the row-by-row statement is
claimed. This uniform-in-$t$ lower bound is what will let us safely
divide by $k$ along each row when constructing the element $q$ in the
proof of Lemma~\ref{lem:h-in-k} below.
\end{remark}

\begin{proof}
By Lemma~\ref{lem:row-idempotents}, $e_j\in J$; and $J\subseteq\sqrt
J=\sqrt{(k)}$ by~\eqref{eq:radical-k}.
Hence $e_j\in\sqrt{(k)}$, i.e.\ (by definition of the radical)
\[
e_j^m\in(k)\qquad\text{for some }m\geq1.
\]
Since $e_j=\chi_{B_j}$ takes only the values $0,1$, it is idempotent,
so $e_j^m=e_j$ for every $m\geq1$; combined with the displayed
membership, this gives $e_j\in(k)$ itself, i.e.\ there exists $b_j\in A$
with
\[
e_j=b_jk.
\]
This is an identity of functions on all of $\nat_{>0}\times\nat_{>0}$;
evaluate it at the points $(j,t)$, $t\geq1$, making up the row $B_j$.
Since $e_j\equiv1$ on $B_j$, this gives
\[
1=b_j(j,t)\,k(j,t)\qquad\text{for every }t\geq1,
\]
so in particular $k(j,t)\neq0$ and $k(j,t)=1/b_j(j,t)$. As
$b_j\in A=\linfty(\nat,\reals)$, we have $|b_j(j,t)|\leq\|b_j\|_\infty$
for every $t$, hence
\[
|k(j,t)|=\frac{1}{|b_j(j,t)|}\geq\frac{1}{\|b_j\|_\infty}
\]
for all $t\geq1$. Taking the infimum over $t$,
\[
\delta_j\geq\frac{1}{\|b_j\|_\infty}>0
\]
(the right-hand side is a well-defined positive number since $b_j\neq0$,
as $b_jk=e_j\neq0$), proving the claim.
\end{proof}

\subsection{The diagonal element}

For each $j$, choose $t_j\geq1$ such that
\[
t_j>j+\max\{0,-\log\delta_j\}.
\]
Let
\[
S=\{(j,t_j):j\geq1\},
\qquad
h=\chi_Sf.
\]

\begin{lemma}\label{lem:h-in-k}
The element $h$ belongs to $(k)$.
\end{lemma}

\begin{remark}
Recall $k=|a_1|+\cdots+|a_n|$ is the single generator of the (assumed
finitely generated) ideal $J$, produced from the hypothetical
compactness of $D(J)$; see~\eqref{eq:radical-k}. Since
$(k)\subseteq\sqrt{(k)}=\sqrt J$, this lemma is exactly what is needed
to place the diagonal element $h$ inside $\sqrt J$: we exhibit an
explicit ``quotient'' $q=h/k$ and check $q\in A$, i.e.\ that $q$ is a
\emph{bounded} function on $\nat_{>0}\times\nat_{>0}$. The only place
division by $k$ is needed is on $S$ (since $h$ vanishes off $S$), and
there $|k|$ is bounded below by $\delta_j$
(Lemma~\ref{lem:row-lower-bound}); the specific choice of $t_j$ made
just above was rigged precisely to keep $|q|$ bounded (indeed decaying)
despite $\delta_j$ possibly shrinking with $j$.
\end{remark}

\begin{proof}
Define
\[
q(j,t)=
\begin{cases}
\dfrac{h(j,t)}{k(j,t)},&k(j,t)\neq0,\\[1ex]
0,&k(j,t)=0.
\end{cases}
\]

\emph{$q$ vanishes off $S$.} If $(j,t)\notin S$ then $h(j,t)=0$
(since $h=\chi_Sf$ is supported on $S$), so $q(j,t)=0$ whether or not
$k(j,t)=0$: in the first case of the definition, $q(j,t)=0/k(j,t)=0$;
in the second, $q(j,t)=0$ by definition.

\emph{Bounding $q$ on $S$.} Fix $j\geq1$ and consider the single point
$(j,t_j)\in S$. By Lemma~\ref{lem:row-lower-bound}, $|k(j,t_j)|\geq
\delta_j>0$ (the infimum defining $\delta_j$ ranges over all $t$, in
particular $t=t_j$); so $k(j,t_j)\neq0$ and we are in the first case of
the definition of $q$. Since $(j,t_j)\in S$, $h(j,t_j)=f(j,t_j)=e^{-t_j}$
(as $h=\chi_Sf$ agrees with $f$ on $S$). Hence
\[
|q(j,t_j)|
=\frac{|h(j,t_j)|}{|k(j,t_j)|}
=\frac{e^{-t_j}}{|k(j,t_j)|}
\leq\frac{e^{-t_j}}{\delta_j}.
\]
By the choice of $t_j$ (namely $t_j>j+\max\{0,-\log\delta_j\}$, which
rearranges to $e^{-t_j}<e^{-j}\delta_j$), this gives
\[
|q(j,t_j)|
\leq\frac{e^{-t_j}}{\delta_j}
<e^{-j}.
\]

\emph{Conclusion.} Combining the two cases, $|q(j,t)|<e^{-j}\leq e^{-1}$
whenever $(j,t)\in S$, and $q=0$ off $S$; so $q$ is bounded, i.e.\
$q\in A=\linfty(\nat,\reals)$. Moreover $h=qk$ pointwise: on $S$ this
holds by the very definition of $q$ (where $k\neq0$), and off $S$ both
sides vanish. Hence
\[
h=qk\in(k)\subseteq\sqrt{(k)}=\sqrt J,
\]
using~\eqref{eq:radical-k} for the last equality.
\end{proof}

\subsection{The diagonal-growth obstruction}

\begin{lemma}\label{lem:h-not-radJ}
The element $h$ does not belong to $\sqrt J$.
\end{lemma}

\begin{proof}
Assume that $h\in\sqrt J$. Then $h^r\in J$ for some $r\geq1$. By the
definition of $J$,
\[
h^rf\in\sqrt{(g)}.
\]
Hence $(h^rf)^m\in(g)$ for some $m\geq1$, so there exists $c\in A$ with
\[
(h^rf)^m=cg.
\]
At the point $(j,t_j)$,
\[
|c(j,t_j)|
=
\frac{e^{-m(r+1)t_j}}{e^{-jt_j}}
=
e^{(j-m(r+1))t_j}.
\]
For $j>m(r+1)$, the exponent is positive; moreover $t_j>j$. Therefore
these values are unbounded as $j\to\infty$, contradicting $c\in\linfty$.
Thus $h\notin\sqrt J$.
\end{proof}

Lemmas~\ref{lem:h-in-k} and~\ref{lem:h-not-radJ} contradict one another under
the assumption that $D(J)$ is compact. Therefore $D(J)$ is not compact.

\begin{theorem}\label{claim:negative-answer}
For
\[
A=\cont{\beta\nat}\cong\linfty(\nat,\reals),
\]
the lattice $\KOp{\Spec{A}}$ is not a Heyting algebra.
\end{theorem}

\begin{proof}
For the elements $f$ and $g$ above,
\[
D(f)\Rightarrow_{\Op{\Spec{A}}}D(g)=D(J)
\]
by Proposition~\ref{prop:frame-implication}. The preceding argument shows
that $D(J)$ is not compact. By Lemma~\ref{lem:implications-coincide},
$\KOp{\Spec{A}}$ cannot be a Heyting algebra.
\end{proof}

\subsection{Generalization: every \(\beta D\) fails}

Nothing above actually used that the whole index set was countable ---
only that a countable grid could be found inside it. So the same
argument covers every $\beta D$ with $D$ infinite discrete, not just
$D=\nat$: enlarging the index set does not rescue a positive answer to
the Bezhanishvili--Tressl question.

\begin{corollary}\label{cor:beta-D}
Let $D$ be an infinite set. Then
\[
\KOp{\Spec{\cont{\beta D}}}
\]
is not a Heyting algebra.
\end{corollary}

\begin{proof}
As in Section~\ref{sec:problem}, restriction to the dense discrete
subspace $D\subseteq\beta D$ gives
$B:=\cont{\beta D}\cong\linfty(D,\reals)$: the proof of
Lemma~\ref{lem:beta-nat-iso} holds verbatim with $D$ in place of $\nat$,
using only that $D$ is dense in $\beta D$ and the universal extension
property of $\beta D$, neither of which needs $D$ countable. Write
$Y=\Spec B$. Lemmas~\ref{lem:bezout-linfty} and
~\ref{lem:compact-opens-principal} hold verbatim with $B$ in place of
$A$: their proofs use only that elements of $\linfty(D,\reals)$ are
bounded real-valued functions on $D$, and never that $D$ is countable.

Fix a countably infinite subset $D_0\subseteq D$ and a bijection
$D_0\cong\nat_{>0}\times\nat_{>0}$; write $d=(j,t)$ for the element of
$D_0$ corresponding to a pair $(j,t)$ under this identification. Define
$f,g\in B$ by
\[
f(d)=e^{-t},
\quad
g(d)=e^{-jt}
\quad\text{for }d=(j,t)\in D_0,
\qquad
f(d)=g(d)=0
\quad\text{for }d\in D\setminus D_0.
\]
Let $J=(\sqrt{(g)}:f)$, and define $B_j$, $e_j$, $S$, and $h$ exactly as
in Section~\ref{sec:proposed-counterexample}; all of these are supported
on $D_0$, since $B_j\subseteq D_0$ for every $j$.

Every step in Lemmas~\ref{lem:row-idempotents}--\ref{lem:h-not-radJ}
checks an identity of one of two kinds: either it is verified pointwise
on $D_0$, where $f,g$ agree exactly with the construction in
Section~\ref{sec:proposed-counterexample}, or it is verified off $D_0$,
where $f=g=0$ and the identity is automatic (for instance $af=0\in
\sqrt{(g)}$ for every $a\in B$, and any witness element such as $b_j$ or
$c$ is entirely unconstrained off $D_0$, since $g=0$ there imposes no
condition on its values). In particular:
\begin{itemize}
\item Lemma~\ref{lem:row-idempotents}'s identity $(e_jf)^j=e_jg$ only
involves values on $B_j\subseteq D_0$, so $e_j\in J$ as before.
\item The lower bound $\delta_j=\inf_{t\ge1}|k(j,t)|>0$ in
Lemma~\ref{lem:row-lower-bound} only uses the restriction of $k$ to
$B_j\subseteq D_0$.
\item The unboundedness obstruction in Lemma~\ref{lem:h-not-radJ} is a
statement about the values of $c$ at the points $(j,t_j)\in D_0$; since
membership in $B=\linfty(D,\reals)$ requires boundedness over \emph{all}
of $D$, unboundedness already witnessed on the subset $D_0$ suffices to
contradict $c\in B$.
\end{itemize}
Hence $D(J)$ is not compact in $Y$, and the proof of
Theorem~\ref{claim:negative-answer} applies verbatim with $A$ replaced by
$B$ and $\Spec A$ replaced by $Y$.
\end{proof}

\section{Beyond \(\beta D\): basically disconnected spaces}
\label{sec:beyond-betaD}

Every $\beta D$ is \emph{atomic}: the clopen algebra of $\beta D$ is
$\mathcal P(D)$, with $D$ itself as a set of atoms, and the diagonal
element $h$ of Section~\ref{sec:proposed-counterexample} is literally
supported on a single point per row. This raises the question whether an
\emph{atomless} compact $F$-space --- for instance a hyperstonean space,
such as the Stone space of the Lebesgue measure algebra on $[0,1]$, which
has no isolated points at all --- escapes the diagonal obstruction. It
does not: the construction survives once ``the point $(j,t_j)$'' is
replaced by ``the cell $C_{j,t_j}$'' throughout, as we show in
Theorem~\ref{claim:beyond-betaD} below.

\begin{definition}[\cite{GillmanJerison}]\label{def:basically-disconnected}
A Tychonoff space $X$ is \emph{basically disconnected} if $\mathrm{cl}(U)$
is open for every cozero set $U=\mathrm{coz}(f)$, $f\in\cont{X}$.
(Every basically disconnected space is zero-dimensional, so this closure
is then automatically clopen.)
\end{definition}

\subsection{Gluing facts}

For a Boolean algebra $\mathcal B$ write $T=\mathrm{Stone}(\mathcal B)$
for its Stone space and $[b]\subseteq T$ for the clopen set corresponding
to $b\in\mathcal B$.

\begin{lemma}\label{lem:finite-gluing}
Let $T$ be compact Hausdorff and let $E_1,\dots,E_N\subseteq T$ be
pairwise disjoint clopen sets. If $\varphi_i\in\cont{E_i}$ for
each $i$, then the function $\varphi$ equal to $\varphi_i$ on $E_i$ and
$0$ elsewhere is continuous on $T$.
\end{lemma}

\begin{proof}
The set $E=E_1\cup\cdots\cup E_N$ is clopen, being a finite union of
clopen sets, and $\varphi|_E$ is continuous because each $E_i$ is open in
$E$ and $\varphi$ is continuous on each $E_i$. Since $T=E\sqcup
(T\setminus E)$ is a partition into two clopen sets, on one of which
$\varphi\equiv0$, $\varphi$ is continuous on $T$.
\end{proof}

\begin{lemma}\label{lem:tail-gluing}
Let $T$ be compact Hausdorff and let $\{E_n\}_{n\geq1}$ be pairwise
disjoint clopen subsets of $T$. If $\varphi_n\in\cont{E_n}$ and
$\sup_{E_n}|\varphi_n|\to0$ as $n\to\infty$, then the function $\varphi$
equal to $\varphi_n$ on $E_n$ for each $n$, and $0$ elsewhere (including
off $\bigcup_nE_n$), is continuous on $T$.
\end{lemma}

\begin{proof}
For $N\geq1$ let $\varphi_N'$ be the function equal to $\varphi_n$ on
$E_n$ for $n\leq N$ and $0$ elsewhere; by Lemma~\ref{lem:finite-gluing},
$\varphi_N'\in\cont{T}$. Since
\[
\|\varphi-\varphi_N'\|_\infty=\sup_{n>N}\sup_{E_n}|\varphi_n|\to0
\]
as $N\to\infty$, $\varphi$ is a uniform limit of continuous functions on
$T$, hence continuous.
\end{proof}

\begin{lemma}\label{lem:sigma-gluing}
Let $T$ be a basically disconnected compact Hausdorff space, so that
$\cont{T}$ is Dedekind $\sigma$-complete as a Riesz space (see
\cite[Thm.~5.6(v)]{bezhanishvili-tressl-2026}). Let $\{b_j\}_{j\geq1}$ be
a countable partition of unity in the Boolean algebra of clopen subsets
of $T$, and let $f_j\in\cont{[b_j]}$ satisfy $f_j\geq0$ and
$\sup_j\|f_j\|_\infty<\infty$. Then there exists $\varphi\in\cont{T}$
with $\varphi\equiv f_j$ on $[b_j]$ for every $j$.
\end{lemma}

\begin{proof}
For $N\geq1$ let
\[
h_N=\sum_{j\leq N}f_j\,\chi_{[b_j]}\in\cont{T}
\]
(continuous by Lemma~\ref{lem:finite-gluing}). The sequence $(h_N)$ is
increasing, since each $f_j\geq0$, and bounded above by the constant
$\sup_j\|f_j\|_\infty$. By Dedekind $\sigma$-completeness, $(h_N)$ has a
supremum $\varphi\in\cont{T}$.

For every $j_0$ and every $N\geq j_0$, $h_N$ restricted to $[b_{j_0}]$
equals $f_{j_0}$ exactly (the other terms of the sum vanish there, the
$[b_j]$ being pairwise disjoint), so, since $\varphi$ is an upper bound
of every $h_N$, $\varphi\geq f_{j_0}$ on $[b_{j_0}]$. It remains to rule
out strict inequality.

Fix $j_0$ and suppose $\varphi(x_0)>f_{j_0}(x_0)$ for some
$x_0\in[b_{j_0}]$. Let $\varphi_1$ agree with $\varphi$ off $[b_{j_0}]$
and with $f_{j_0}$ on $[b_{j_0}]$; by Lemma~\ref{lem:finite-gluing}
(applied to the two clopen pieces $[b_{j_0}]$ and $T\setminus[b_{j_0}]$),
$\varphi_1\in\cont{T}$. For every $N$, $h_N\leq\varphi_1$: off
$[b_{j_0}]$ this is inherited from $h_N\leq\varphi$; on $[b_{j_0}]$,
$h_N\leq f_{j_0}=\varphi_1$ because $h_N$ restricted to $[b_{j_0}]$ is
either $0$ or $f_{j_0}$. So $\varphi_1$ is an upper bound of $(h_N)$ with
$\varphi_1\leq\varphi$ and $\varphi_1(x_0)<\varphi(x_0)$, contradicting
that $\varphi$ is the \emph{least} upper bound. Hence $\varphi=f_j$ on
$[b_j]$ for every $j$.
\end{proof}

\subsection{Basic disconnectedness from \(\sigma\)-completeness}

\begin{definition}[\cite{GillmanHenriksen1956}]\label{def:F-space}
A Tychonoff space, $X$, is said to be an $F$-space if $\cont{X}$ is B\'ezout.
\end{definition}

\begin{remark}
\cite{bezhanishvili-tressl-2026} instead \emph{defines} $T$ to be an
$F$-space by requiring $\Spec{\cont{T}}$ to be stranded, i.e.\ a
forest under specialization. The two definitions are equivalent
--- this is the classical content of \cite[Theorem in
14.25]{GillmanJerison} --- so Definition~\ref{def:F-space} above and the
notion of $F$-space used in \cite{bezhanishvili-tressl-2026} pick out the
same class of spaces; we use the Bézout formulation throughout because it
is what feeds directly into Theorem~\ref{thm:bezout-criterion}, which is
how $F$-spaces enter our arguments (e.g.\ in the proof of
Theorem~\ref{claim:beyond-betaD}).
\end{remark}

Theorem~\ref{claim:beyond-betaD} below is stated for an abstract
$\sigma$-complete Boolean algebra $\mathcal B$, but its proof needs to
know that $T=\mathrm{Stone}(\mathcal B)$ is basically disconnected (to
invoke that basically disconnected spaces are $F$-spaces --- hence
Bézout --- and to invoke Lemma~\ref{lem:sigma-gluing}). The following
closes that gap.

\begin{lemma}\label{lem:sigma-complete-implies-bd}
If $\mathcal B$ is a $\sigma$-complete Boolean algebra, then
$T=\mathrm{Stone}(\mathcal B)$ is basically disconnected.
\end{lemma}

\begin{proof}
Let $h\in\cont{T}$ and $U=\mathrm{coz}(h)$; we show
$\mathrm{cl}(U)$ is clopen.

For $n\geq1$, $\{x:|h(x)|\geq1/n\}$ is closed, hence compact. Since $T$
is zero-dimensional and $h$ is continuous, every point of this compact
set has a clopen neighbourhood on which $|h|>1/(n+1)$; by compactness,
finitely many such neighbourhoods cover $\{|h|\geq1/n\}$. Let $E_n$ be
their (clopen) union, so
\[
\{|h|\geq1/n\}\subseteq E_n\subseteq\{|h|>1/(n+1)\}\subseteq U.
\]
Since $h(x)\neq0$ iff $|h(x)|\geq1/n$ for some $n$, taking the union
over $n$ gives $U=\bigcup_nE_n$: a countable union of clopen sets.

By $\sigma$-completeness, $F:=\bigvee_nE_n$ exists in the Boolean
algebra of clopen subsets of $T$. As $F$ is closed and contains
$\bigcup_nE_n=U$, $F\supseteq\mathrm{cl}(U)$. Conversely, suppose
$x\in F$ but $x\notin\mathrm{cl}(U)$; by zero-dimensionality there is a
clopen $V\ni x$ with $V\cap U=\emptyset$, hence $V\cap E_n=\emptyset$
for every $n$, i.e.\ $E_n\leq V^c$ for every $n$. Then $F\wedge V^c$ is
also a clopen upper bound of $\{E_n\}$, and $F\wedge V^c<F$ strictly
(since $x\in F\cap V$, so $F\wedge V\neq0$), contradicting that $F$ is
the \emph{least} such upper bound. Hence $F\subseteq\mathrm{cl}(U)$, so
$F=\mathrm{cl}(U)$.

Thus $\mathrm{cl}(U)=F$ is clopen, in particular open. As $U$ was an
arbitrary cozero set, $T$ is basically disconnected.
\end{proof}

This is the converse of Lemma~\ref{lem:clopen-sigma-complete} below
(which shows the other direction, for the clopen algebra of a basically
disconnected space); together the two give an if-and-only-if.

\subsection{The point-versus-cell obstruction}

In the construction of Section~\ref{sec:proposed-counterexample}, the
diagonal set $S=\{(j,t_j):j\geq1\}$ consists of a single point per row,
and $h=\chi_Sf$ is supported there. If $\mathcal B$ is atomless, its
Stone space has no isolated points, so no nonzero element of
$\cont{T}$ can be supported at a single point: any function
vanishing on every neighbourhood of a non-isolated point $x$ except at
$x$ itself is discontinuous at $x$ unless its value there is $0$. The
point-supported diagonal element therefore has no atomless analogue.
Replacing ``the point $(j,t_j)$'' by ``the cell $C_{j,t_j}$'' --- a
non-degenerate clopen piece, not a point --- repairs the construction, as
we now show.

\subsection{The generalized construction}

\begin{theorem}\label{claim:beyond-betaD}
Let $\mathcal B$ be an infinite $\sigma$-complete Boolean algebra and let
$T=\mathrm{Stone}(\mathcal B)$. Suppose $\mathcal B$ admits a countable
grid: a partition of unity $\{B_j\}_{j\geq1}$ together with, for each
$j$, a countable partition $\{C_{j,t}\}_{t\geq1}$ of $B_j$, all pieces
nonzero. Then $\KOp{\Spec{\cont{T}}}$ is not a Heyting algebra.
\end{theorem}

\begin{proof}
Write $A=\cont{T}$. Since $\mathcal B$ is $\sigma$-complete,
$T=\mathrm{Stone}(\mathcal B)$ is basically disconnected
(Lemma~\ref{lem:sigma-complete-implies-bd}), hence an $F$-space, so $A$
is Bézout; and $\cont{T}$ is Dedekind $\sigma$-complete, so
Lemma~\ref{lem:sigma-gluing} applies with this $T$.

\emph{Construction of $f,g$.} For each $j$, define
$f_j,g_j\in\cont{[B_j]}$ by $f_j\equiv e^{-t}$, $g_j\equiv
e^{-jt}$ on $[C_{j,t}]$; both are nonnegative and bounded by $1$, and
since these values decay in $t$, Lemma~\ref{lem:tail-gluing} shows
$f_j,g_j$ are well-defined continuous functions on $[B_j]$. The values
of $\{f_j\}_j$ (respectively $\{g_j\}_j$) are uniformly bounded by $1$
but do not decay as $j\to\infty$ (e.g.\ $f_j\equiv e^{-1}$ on
$[C_{j,1}]$ for every $j$), so gluing them across the partition of
unity $\{B_j\}_j$ needs Lemma~\ref{lem:sigma-gluing} in full strength;
it produces $f,g\in A$ with $f\equiv f_j$, $g\equiv g_j$ on $[B_j]$ for
every $j$ --- that is, $f\equiv e^{-t}$ and $g\equiv e^{-jt}$ on
$[C_{j,t}]$ for all $j,t$.

\emph{Row idempotents.} With $e_j:=\chi_{[B_j]}$ and
$J:=(\sqrt{(g)}:f)$: since $\bigvee_tC_{j,t}=B_j$, the set
$\bigcup_t[C_{j,t}]$ is dense in $[B_j]$ --- if a nonempty clopen
$[a]\subseteq[B_j]$ met no $[C_{j,t}]$, then $a\wedge C_{j,t}=0$ for
every $t$, so $a=a\wedge B_j=a\wedge\bigvee_tC_{j,t}
=\bigvee_t(a\wedge C_{j,t})=0$ (countable distributivity, as in
Theorem~\ref{thm:grid-existence}'s Step~1), a contradiction. Since
$f^j$ and $g$ are continuous and agree on this dense subset (both equal
$e^{-jt}$ on $[C_{j,t}]$), they agree on all of $[B_j]$. As in
Lemma~\ref{lem:row-idempotents}, $(e_jf)^j=e_jf^j=e_jg\in(g)$, so
$e_j\in J$.

\emph{Row lower bound.} Assuming $D(J)$ compact gives $k\in A$ with
$\sqrt J=\sqrt{(k)}$, as in
Section~\ref{sec:proposed-counterexample}. From $e_j\in J\subseteq
\sqrt{(k)}$ and idempotence of $e_j$, $e_j=b_jk$ for some $b_j\in A$;
this identity of continuous functions holds at every point of $[B_j]$,
giving $\delta_j:=\inf_{[B_j]}|k|\geq1/\|b_j\|_\infty>0$.

\emph{The diagonal element.} Choose $t_j$ as in
Section~\ref{sec:proposed-counterexample} and set $h:=f$ on
$\bigcup_jC_{j,t_j}$, $0$ elsewhere; this lies in $A$ by
Lemma~\ref{lem:tail-gluing}, since $|f|=e^{-t_j}\to0$ on the chosen
cells. Define $q:=e^{-t_j}/k$ on $[C_{j,t_j}]$ (continuous and bounded
there, since $k$ is nonvanishing and bounded away from $0$ by $\delta_j$
on the larger set $[B_j]\supseteq[C_{j,t_j}]$), $0$ elsewhere; again $q\in
A$ by Lemma~\ref{lem:tail-gluing}. Then $h=qk\in(k)\subseteq\sqrt{(k)}=
\sqrt J$.

\emph{Growth obstruction.} Suppose $h\in\sqrt J$, so $h^r\in J$,
$h^rf\in\sqrt{(g)}$, and $(h^rf)^m=cg$ for some $c\in A$. On
$[C_{j,t_j}]$, $h$ and $f$ both equal the constant $e^{-t_j}$, and $g$
equals the constant $e^{-jt_j}$, so this identity forces $c$ to equal the
constant $e^{(j-m(r+1))t_j}$ on all of $[C_{j,t_j}]$. For $j>m(r+1)$ this
is positive and, since $t_j>j$, unbounded as $j\to\infty$, contradicting
$c\in A$.

As in Section~\ref{sec:proposed-counterexample}, this contradiction shows
$D(J)$ is not compact, and Proposition~\ref{prop:frame-implication}
together with Lemma~\ref{lem:implications-coincide} then show
$\KOp{\Spec A}$ is not a Heyting algebra.
\end{proof}

\begin{corollary}\label{cor:hyperstonean}
Let $\mathcal M$ be the Lebesgue measure algebra on $[0,1]$ (measurable
sets modulo null sets) and let $T=\mathrm{Stone}(\mathcal M)$ be the
associated hyperstonean space, so that $\cont{T}\cong
L^\infty([0,1],\reals)$. Then $\KOp{\Spec{\cont{T}}}$ is not a
Heyting algebra.
\end{corollary}

\begin{proof}
$\mathcal M$ is atomless, hence admits a countable grid: partition
$[0,1]$ into rows $R_j=[1-2^{1-j},1-2^{-j})$, $j\geq1$, and partition
each row into columns $R_{j,t}\subseteq R_j$ of positive measure,
$t\geq1$, with $\bigcup_tR_{j,t}=R_j$ up to a null set. Countable joins
in the measure algebra are computed by taking unions of representatives,
so the classes $[R_j]$ and $[R_{j,t}]$ satisfy
$\bigvee_j[R_j]=1_{\mathcal M}$ and $\bigvee_t[R_{j,t}]=[R_j]$ exactly
(the rows differ from $[0,1]$ only by the null set $\{1\}$, and each
row's columns differ from it only by a null set), so this is a genuine
grid in the sense required by Theorem~\ref{claim:beyond-betaD}. Apply
that theorem.
\end{proof}

\section{The full negative resolution}
\label{sec:full-negative}

Theorem~\ref{claim:beyond-betaD} still needs a countable grid to exist.
This section removes that hypothesis: every infinite $\sigma$-complete
Boolean algebra admits one (Theorem~\ref{thm:grid-existence} below), and
every infinite basically disconnected compact Hausdorff space has clopen
algebra of exactly this type (Lemmas~\ref{lem:clopen-sigma-complete}
and~\ref{lem:clopen-infinite}). Together with the chain Esakia
$\Rightarrow$ pseudocomplemented $\Rightarrow$ basically disconnected
already established in \cite[Thm.~5.6]{bezhanishvili-tressl-2026}, this
settles the open problem of \cite{bezhanishvili-tressl-2026}: no infinite
compact Hausdorff space $T$ has $\Spec{\cont{T}}$ Esakia
(Theorem~\ref{thm:full-negative}).

\subsection{Grid existence}

\begin{theorem}\label{thm:grid-existence}
Every infinite $\sigma$-complete Boolean algebra $\mathcal B$ admits a
countable grid: a partition of unity $\{B_j\}_{j\geq1}$ (pairwise
disjoint, nonzero, $\bigvee_jB_j=1$) together with, for each $j$, a
countable partition $\{C_{j,t}\}_{t\geq1}$ of $B_j$ (pairwise disjoint,
nonzero, $\bigvee_tC_{j,t}=B_j$).
\end{theorem}

\begin{proof}
\emph{Step 0: an unconditional antichain.} Since $\mathcal B$ is infinite
it has more than two elements; pick $0<c_1<1$ and set $d_1:=c_1^c$ (both
nonzero). Since $\mathcal B\cong{\downarrow}c_1\times{\downarrow}d_1$ and
$\mathcal B$ is infinite, at least one factor is infinite; relabel so
that ${\downarrow}c_1$ is infinite. Recursing inside ${\downarrow}c_1$:
pick $0<c_2<c_1$ with ${\downarrow}c_2$ infinite, and set
$d_2:=c_1\wedge c_2^c$. Iterating gives a strictly decreasing sequence
$c_1>c_2>\cdots$ with ${\downarrow}c_n$ infinite at each stage, and
$d_n:=c_{n-1}\wedge c_n^c$ (with $c_0:=1$) pairwise disjoint (if $m>n$
then $d_m\leq c_{m-1}\leq c_n$, so
$d_m\wedge d_n\leq c_n\wedge c_{n-1}\wedge c_n^c=0$) and nonzero by
construction. This gives an antichain $\{d_n\}_{n\geq1}$, valid in any
infinite Boolean algebra, with no completeness assumption.

\emph{Step 1: countable distributivity.} For $x\in\mathcal B$ and
countable $\{y_n\}\subseteq\mathcal B$,
\[
x\wedge\bigvee_ny_n=\bigvee_n(x\wedge y_n)
\]
(both joins exist by $\sigma$-completeness). Writing $s=\bigvee y_n$,
$s'=\bigvee(x\wedge y_n)$: $s'\leq x\wedge s$ is immediate. For the
converse, $x\wedge s\leq s'$ is equivalent to $s\leq s'\vee x^c$; and for
each $n$, $y_n=(y_n\wedge x)\vee(y_n\wedge x^c)\leq s'\vee x^c$ (since
$y_n\wedge x\leq x\wedge y_n\leq s'$), so $s'\vee x^c$ is an upper bound
of $\{y_n\}$, giving $s\leq s'\vee x^c$ since $s$ is the least such.
Hence $x\wedge s\leq s'$, so $x\wedge s=s'$.

\emph{Step 2: two cases by number of atoms.} Let $\mathrm{At}(\mathcal
B)$ be the set of atoms of $\mathcal B$.

\emph{Case A: $\mathrm{At}(\mathcal B)$ infinite.} Choose a countably
infinite subset $\{p_1,p_2,\dots\}\subseteq\mathrm{At}(\mathcal B)$ and
partition it into countably many countably infinite pieces
$P_1,P_2,\dots$ via any bijection $\nat\cong\nat\times\nat$. By
$\sigma$-completeness, $B_j:=\bigvee_{p\in P_j}p$ exists. Then $B_j\neq0$
(it dominates any $p\in P_j$); for $j\neq k$,
\[
B_j\wedge B_k=\bigvee_{p\in P_j}(p\wedge B_k)
=\bigvee_{p\in P_j}\bigvee_{q\in P_k}(p\wedge q)=0
\]
by Step~1 (distinct atoms are disjoint); ${\downarrow}B_j$ is infinite,
since distinct finite $F\neq F'\subseteq P_j$ give distinct joins
$\bigvee F\neq\bigvee F'$ (an atom $p\in P_j$ satisfies $p\leq\bigvee F$
iff $p\in F$, again by Step~1), and $P_j$ is countably infinite; and the
atoms of ${\downarrow}B_j$ are exactly $P_j$ (if $q$ is an atom
$\leq B_j$, then $q=q\wedge B_j=\bigvee_{p\in P_j}(q\wedge p)\neq0$, so
$q\wedge p\neq0$ for some $p\in P_j$, forcing $q=p$). Finally,
${\downarrow}B_j$ is $\sigma$-complete: any countable
$\{y_n\}\subseteq{\downarrow}B_j$ has $\bigvee y_n$ (computed in
$\mathcal B$) satisfying $\bigvee y_n\leq B_j$, since $B_j$ is an upper
bound of $\{y_n\}$ and $\bigvee y_n$ is the least such. So
${\downarrow}B_j$ is again an infinite $\sigma$-complete Boolean algebra
with countably many atoms; Case~A recurses inside it to produce a
countably infinite antichain below $B_j$, each member again of infinite
relativization.

\emph{Case B: $\mathrm{At}(\mathcal B)$ finite (possibly empty).} Let
$e:=\bigvee\mathrm{At}(\mathcal B)$, a finite join. A direct computation
($x=\bigvee_i(x\wedge p_i)$ for $x\leq e$, the $p_i$ the finitely many
atoms) shows ${\downarrow}e$ consists exactly of the sub-joins of the
$p_i$, so it is finite. Since $\mathcal B\cong{\downarrow}e\times
{\downarrow}e^c$ and $\mathcal B$ is infinite, ${\downarrow}e^c$ is
infinite; and since every atom of $\mathcal B$ is $\leq e$,
${\downarrow}e^c$ is atomless. In an atomless algebra every nonzero
$x\leq e^c$ is itself atomless (an atom $\leq x$ would be an atom
$\leq e^c$) and hence, being a nonzero atomless element, never an atom
and always of infinite relativization (a finite nonzero Boolean algebra
is always atomic). So Step~0's peeling, run inside ${\downarrow}e^c$,
produces a countably infinite antichain $\{d_n\}$ below $e^c$, every
member of which is again atomless (hence again infinite), and the same
argument recurses with no case distinction needed. This sub-case uses no
$\sigma$-completeness, only finitely many splits repeated countably
often.

\emph{Step 3: from antichain to partition.} Steps~0--2 produce, below
any given $e\in\mathcal B$ (with $e=1$ at the top level, or $e=B_j$ when
recursing inside ${\downarrow}B_j$), a countably infinite antichain
$\{E_n\}_{n\geq1}$ of pairwise disjoint nonzero elements each of
infinite relativization --- but the antichain need not itself join to
$e$. To repair this, set
\[
s:=\bigvee_nE_n\quad(\text{exists by }\sigma\text{-completeness}),
\qquad
r:=e\wedge s^c,
\]
and define $E_1':=E_1\vee r$, $E_n':=E_n$ for $n\geq2$. Since $r\leq
s^c$ is disjoint from every $E_n\leq s$, and the $E_n$ are themselves
pairwise disjoint, $\{E_n'\}$ is again pairwise disjoint; each $E_n'$ is
nonzero and has infinite relativization (as $E_n'\geq E_n$). Since
$E_n\leq e$ for every $n$, also $s\leq e$, so $e\wedge s=s$ and
\[
\bigvee_nE_n'=r\vee\bigvee_nE_n=r\vee s=(e\wedge s^c)\vee(e\wedge s)=e.
\]
So $\{E_n'\}$ is a genuine countable partition of $e$ into pieces of
infinite relativization.

Applying Step~3 with $e=1$ to the antichain from Step~2 gives the
partition of unity $\{B_j\}_{j\geq1}$. Applying Step~3 again, with
$e=B_j$, to the antichain that Step~2 recurses to inside
${\downarrow}B_j$, gives the partition $\{C_{j,t}\}_{t\geq1}$ of $B_j$.
This is the required grid.
\end{proof}

\begin{remark}
The $\sigma$-completeness hypothesis is doing real work: the Boolean
algebra of finite and cofinite subsets of $\nat$ has infinitely many
atoms (the singletons) but no countable grid, since any pairwise-disjoint
family contains at most one cofinite member (two disjoint cofinite sets
is impossible), so every other member is finite and hence terminal. This
algebra is not $\sigma$-complete: the countable family of odd singletons
$\{1\},\{3\},\{5\},\dots$ has no least upper bound in it (any cofinite
upper bound can be shrunk further while staying cofinite; the set of all
odd numbers, the ``true'' supremum, is neither finite nor cofinite).
\end{remark}

\subsection{From an abstract algebra to a concrete space}

\begin{lemma}\label{lem:clopen-sigma-complete}
If $T$ is basically disconnected, then the Boolean algebra of clopen
subsets of $T$ is $\sigma$-complete.
\end{lemma}

\begin{proof}
Let $\{E_n\}_{n\geq1}$ be any countable family of clopen subsets of $T$
(no disjointness assumed). Each $\chi_{E_n}\in\cont{T}$ (locally
constant on the two clopen pieces $E_n$, $T\setminus E_n$). Define
$f:=\sum_n2^{-n}\chi_{E_n}$; the partial sums are continuous, and
$\|f-\sum_{n\leq N}2^{-n}\chi_{E_n}\|_\infty\leq\sum_{n>N}2^{-n}\to0$, so
$f$ is a uniform limit of continuous functions, hence $f\in\cont{T}$.
Since $f(x)=\sum_{n:x\in E_n}2^{-n}$ is a sum of strictly positive terms,
$f(x)=0$ iff $x\notin E_n$ for every $n$, so $\mathrm{coz}(f)=\bigcup_nE_n$
exactly: $\bigcup_nE_n$ is a cozero set. Basic disconnectedness then gives
that $\mathrm{cl}(\bigcup_nE_n)$ is open, hence (being a closure) clopen.
Any clopen $F\supseteq E_n$ for every $n$ is closed and contains
$\bigcup_nE_n$, hence contains $\mathrm{cl}(\bigcup_nE_n)$; so
$\mathrm{cl}(\bigcup_nE_n)$ is the least such, i.e.\ it is the supremum of
$\{E_n\}$ in the Boolean algebra of clopen sets.
\end{proof}

\begin{lemma}\label{lem:clopen-infinite}
If $T$ is infinite and zero-dimensional (in particular, if $T$ is
basically disconnected and infinite), then $T$ has infinitely many
clopen subsets.
\end{lemma}

\begin{proof}
Suppose toward a contradiction that $T$ has finitely many clopen
subsets. Since clopens form a basis, the whole topology is generated by
a finite collection, hence has at most finitely many open sets. For
$x\in T$, let $Q_x$ be the intersection of all clopen sets containing
$x$; this is a finite intersection, hence itself clopen. If some $y\neq
x$ lay in $Q_x$, Hausdorffness gives disjoint opens $U\ni x$, $V\ni y$,
and since clopens form a basis there is a clopen $E$ with $x\in
E\subseteq U$; then $y\notin E$ (as $E\subseteq U$ and $U\cap
V=\emptyset\ni y$), so $y\notin Q_x$ (as $Q_x\subseteq E$), a
contradiction. So $Q_x=\{x\}$: every singleton is clopen, i.e.\ every
point of $T$ is isolated. A compact space in which every point is
isolated is finite (the cover by singletons has no proper subcover),
contradicting $T$ infinite. Hence $T$ has infinitely many clopen
subsets.
\end{proof}

\subsection{The theorem}

\begin{theorem}\label{thm:full-negative}
For every infinite compact Hausdorff space $T$, the lattice
$\KOp{\Spec{\cont{T}}}$ is not a Heyting algebra; equivalently,
$\Spec{\cont{T}}$ is not an Esakia space. This is a full negative
answer to the open problem of \cite{bezhanishvili-tressl-2026}.
\end{theorem}

\begin{proof}
Split on whether $T$ is basically disconnected.

If $T$ is not basically disconnected: by
\cite[Thm.~5.6]{bezhanishvili-tressl-2026}, $T$ basically disconnected is
equivalent to $\Spec{\cont{T}}$ being a PC-space (i.e.\
$\KOp{\Spec{\cont{T}}}$ pseudocomplemented). In any Heyting
algebra, the pseudocomplement $a^\ast:=a\Rightarrow0$ always exists, so
Heyting implies pseudocomplemented; hence Esakia implies PC-space
implies $T$ basically disconnected. Contrapositively, $T$ not basically
disconnected implies $\Spec{\cont{T}}$ not Esakia.

If $T$ is basically disconnected (and, by hypothesis, infinite compact
Hausdorff): by Lemmas~\ref{lem:clopen-sigma-complete}
and~\ref{lem:clopen-infinite}, $\mathcal B$, the Boolean algebra of
clopen subsets of $T$, is an infinite $\sigma$-complete Boolean algebra.
By Theorem~\ref{thm:grid-existence}, $\mathcal B$ admits a countable
grid. By Stone duality for compact Hausdorff zero-dimensional spaces,
$T\cong\mathrm{Stone}(\mathcal B)$. Theorem~\ref{claim:beyond-betaD} then
applies directly, with this $\mathcal B$ and this grid, giving
$\KOp{\Spec{\cont{T}}}$ not a Heyting algebra.

The two cases are exhaustive.
\end{proof}

\section{Examples}

The criterion obtained in Theorem~\ref{thm:bezout-criterion} is nontrivial even for spectra that are not
Stone spaces.

\begin{example}
Let
$
R=\integers.
$
Then \(R\) is a principal ideal domain, hence a B\'ezout ring.
Moreover, for every \(g\in R\),
\[
\sqrt{(g)}=(\operatorname{rad}(g)),
\]
where \(\operatorname{rad}(g)\) denotes the square-free part of \(g\).

Since every ideal of a principal ideal domain is principal, for every
\(f,g\in R\) the ideal
\[
(\sqrt{(g)}:f)
\]
is principal, and therefore has principal radical.
By Theorem~\ref{thm:bezout-criterion},
\[
\KOp{\Spec{\integers}}
\]
is a Heyting algebra.

On the other hand,
$
\Spec{\integers}
$
is not Hausdorff. Indeed, the generic point
$
(0)\in\Spec{\integers}
$
belongs to every nonempty open subset, so the space is not Hausdorff.
\end{example}

\printbibliography

@misc{bezhanishvili-tressl-2026,
      title={Pseudocomplementation in rings of continuous functions}, 
      author={Guram Bezhanishvili and Marcus Tressl},
      year={2026},
      eprint={2603.28165},
      archivePrefix={arXiv},
      primaryClass={math.GN},
      url={https://arxiv.org/abs/2603.28165}, 
}

@book{GillmanJerison,
  author    = {Leonard Gillman and Meyer Jerison},
  title     = {Rings of Continuous Functions},
  publisher = {Springer-Verlag},
  address   = {New York},
  series    = {Graduate Texts in Mathematics},
  volume    = {43},
  year      = {1976},
  edition   = {Reprint of the 1960 edition}
}

@article{GillmanHenriksen1956,
  author  = {Leonard Gillman and Melvin Henriksen},
  title   = {Rings of Continuous Functions in Which Every Finitely Generated Ideal Is Principal},
  journal = {Transactions of the American Mathematical Society},
  volume  = {82},
  number  = {2},
  pages   = {366--391},
  year    = {1956}
}

\end{document}